\documentclass[reqno,12pt]{amsart}

\usepackage{amsmath, amsthm, amsfonts, amssymb}
\input{mathrsfs.sty}

\usepackage{amsmath}
\usepackage{amsfonts}
\usepackage{amssymb}
\usepackage{eufrak}
\usepackage{amscd}
\usepackage{amsthm}
\usepackage{amstext}
\usepackage[all]{xy}

\newcommand{\id}{\operatorname{id}}

   \theoremstyle{plain}
   \newtheorem{thm}{Theorem}[section]
   \newtheorem{prop}[thm]{Proposition}
   \newtheorem{lem}[thm]{Lemma}
   \newtheorem{cor}[thm]{Corollary}
   \theoremstyle{definition}
   
   \newtheorem{defn}[thm]{Definition}
   \newtheorem{example}[thm]{Example}
   \theoremstyle{remark}

 \numberwithin{equation}{section}

\author{V. Manuilov}

\date{}

\address{Moscow Center for Fundamental and Applied Mathematics, Moscow State University,
Leninskie Gory 1, Moscow, 
119991, Russia}

\email{manuilov@mech.math.msu.su}

\title[Inverse semigroups of metrics on doubles related to subsets]{Inverse semigroups of metrics on doubles related to certain subsets}

\begin{document}

\begin{abstract}
Recently we have shown that the equivalence classes of metrics on the double of a metric space $X$ form an inverse semigroup. Here we define an inverse subsemigroup related to a family of isometric subspaces of $X$, which is more computable. As a special case, we study this subsemigroup related to the family of geodesic rays starting from the basepoint, for Euclidean spaces and for trees. 

\end{abstract}

\maketitle

\section{Introduction}

Recently, in \cite{M1}, we have shown that the equivalence classes of metrics on the double of a metric space $X$ form an inverse semigroup. Here we define an inverse subsemigroup related to a family of isometric subspaces of $X$, which is more computable. As a special case, we study this subsemigroup related to the family of geodesic rays starting from the basepoint, for Euclidean spaces and for trees. 

Let $X=(X,d_X)$ be a metric space with the metric $d_X$. 
We say that a metric $\rho$ on $X\sqcup X$ is {\it compatible} with $d_X$ if both restrictions of $\rho$ onto each copy of $X$ coincide with $d_X$. Recall that metrics $d$ and $d'$ on $Y$ are coarsely equivalent (we write $d\sim_c d'$) if there exists a monotone homeomorphism $\varphi$ of $(0,\infty)$ such that $d(y_1,y_2)\leq \varphi(d'(y_1,y_2))$ and $d'(y_1,y_2)\leq\varphi(d(y_1,y_2))$ for any $y_1,y_2\in Y$ \cite{Roe}. Let $M(X)$ be the set of coarse equivalence classes of metrics on the double $X\sqcup X$ of $X$ compatible with $d_X$. 

It was shown in \cite{M1} that $M(X)$ has a natural structure of an inverse semigroup. Here we recall this structure.
For convenience, we identify $X\sqcup X$ with $X\times\{0,1\}$ and write $x$ for $(x,0)$ and $x'$ for $(x,1)$, where $x\in X$.
Composition of two metrics, $\rho$ and $\sigma$, on $X\sqcup X$, compatible with $d_X$, was defined in \cite{M0} by 
$$
\rho\circ\sigma(x,y')=\inf_{u\in X}[\sigma(x,u')+1+\rho(u,y')];
$$
$$
\rho\circ\sigma(x,y)=\rho\circ\sigma(x',y')=d_X(x,y),\qquad x,y\in X. 
$$
Clearly, if $\rho_1\sim_c\rho_2$, $\sigma_1\sim_c\sigma_2$ then $\rho_1\circ\sigma_1\sim_c\rho_2\circ\sigma_2$.
We write $[\rho]$ for the coarse equivalence class of a metric $\rho$. It was shown in \cite{M1} that $M(X)$ is an inverse semigroup with respect to this composition, with the pseudoinverse element for $s=[\rho]\in M(X)$ given by $s^*=[\rho^*]$, where $\rho^*$ is a metric on $X\sqcup X$ compatible with $d_X$ such that $\rho^*(x,y')=\rho(y,x')$ for any $x,y\in X$.

The deficiency of the inverse semigroup $M(X)$ is that it is too great, as Lemma \ref{mnogo} below shows.
Recall that, given a subset $A\subset X$, there is a metric $\rho_A$ compatible with $d_X$ such that
\begin{equation}\label{rho}
\rho_A(x,y')=\inf_{a\in A}[d_X(x,a)+1+d_X(a,y)],
\end{equation}
and $[\rho_A]=[\rho_B]$ iff the Hausdorff distance between $A$ and $B$ is finite \cite{M1}.

\begin{lem}\label{mnogo}
Let $X=\mathbb R_+=[0,\infty)$ be the set of non-negative reals with the standard metric. Then the cardinality of $M(X)$ is at least continuum.

\end{lem}
\begin{proof}
Let $\lambda\in(1,\infty)$, $A_\lambda=\{\lambda^n:n\in\mathbb N\}$. Clearly, the Hausdorff distance between $A_\lambda$ and $A_\mu$ is infinite when $\mu\neq\lambda$.

\end{proof}
 
Our aim is to define an inverse subsemigroup from metrics on the double of $X$ that is substantially smaller and more computable than $M(X)$.

The idea is to find in $X$ a family of isometric subsets and to find an inverse subsemigroup in $M(X)$ that surjects onto the semigroup of partial bijections of this family. In Section 2 we describe the construction of such subsemigroup for a general inverse semigroup and a family of its idempotents. In Section 3 we show how this construction works for the semigroup $M(X)$. One of the most natural choices for a family of idempotent metrics is that related to geodesic rays, and in Section 4 we give some details for the case when the family of isometric subsets of a geodesic metric space $X$ consists of geodesic rays starting at the fixed point $x_0\in X$. Sections 5 and 6 deal with the cases when $X$ is a Euclidean space and a rooted tree, respectively. In these two cases we evaluate the image of this subsemigroup in the semigroup of partial bijections of the visual boundary of $X$.

\section{An inverse subsemigroup related to a set of idempotents}\label{semigroup}

Let $S$ be an inverse semigroup with $0$ and $1$, $E(S)$ its set of idempotents. Recall that a semigroup is inverse if any element $s\in S$ has a unique pseudoinverse $t\in S$, which means that $sts=s$ and $tst=t$. This pseudoinverse element is denoted by $s^*$. Any two elements in $E(S)$ commute, $s^*s$ is idempotent for any $s\in S$, and $E(S)$ is partially ordered: for $e,f\in E(S)$, $e\leq f$ if $ef=e$. Two idempotents, $e$ and $f$ are equivalent if there exists $s\in S$ such that 
$ss^*=e$ and $s^*s=f$. In this case $es=ss^*s=sf$. 
Fix a set $\Phi\subset E(S)$ of mutually equivalent non-zero idempotents. Our aim is to define an inverse subsemigroup in $S$ that surjects onto the inverse semigroup of partial bijections of $\Phi$.

Recall that, given a set $X$, a partial bijection is a triple $(A,f,B)$, where $A,B\subset X$ and $f:A\to B$ is a bijection. The set $PB(X)$ of all partial bijiections of $X$ is an example of an inverse semigroup.


Consider the following subset $S(\Phi)\subset S$. Let $s\in S(\Phi)$ if, for any $e\in\Phi$, one of the following holds: 
\begin{itemize}
\item
either $es^*s=0$; 
\item
or $es^*s=e$ and then there exists $f\in\Phi$ such that 
\begin{equation}\label{1}
se=fs\quad \mbox{and}\quad fss^*=f. 
\end{equation}
\end{itemize}

Note that such $f$ is unique. Indeed, if $f_1,f_2\in \Phi$ satisfy (\ref{1}) then $f_1s=f_2s=se$. Multiplying this by $f_1$, we get $f_1s=f_1^2s=f_1f_2s$, hence $f_1=f_1ss^*=f_1f_2ss^*=f_1f_2$, which means that $f_1\leq f_2$. Similarly, $f_2\leq f_1$, hence $f_1=f_2$. 

Put $S_\Phi=S(\Phi)\cap S(\Phi)^*$, where $R^*=\{s^*:s\in R\}$.

\begin{lem}
$S_\Phi$ is an inverse semigroup.

\end{lem}
\begin{proof}
Let $s,t\in S(\Phi)$, and let $e\in\Phi$. If $et^*t=0$ then, using that $s^*s\leq 1$ for any $s\in S$, we have $et^*s^*st\leq et^*t=0$, hence $e(st)^*st=0$. If $et^*t\neq 0$ then there exists $f\in\Phi$ such that $te=ft$. Then $et^*=t^*f$, hence $e(st)^*st=t^*fs^*st$. Now, there are two possibilities. Either $fs^*s=0$, and then $e(st)^*st=t^*0t=0$, or $fs^*s=f$ and there exists $g\in\Phi$ such that $sf=gs$ and $gss^*=g$. In the latter case, we have $(st)e=sft=g(st)$. We also have 
$$
e(st)^*st=et^*s^*st=t^*fs^*st=t^*ft=t^*te=e, 
$$
and, similarly, $g(st)(st)^*=g$.
This proves that $st\in S(\Phi)$.  

By definition, if $s\in S_\Phi$ then $s^*\in S_\Phi$, hence the semigroup $S_\Phi$ is inverse. 

\end{proof}  

Let $s\in S_\Phi$. Set 
\begin{equation}\label{As}
A(s)=\{e\in\Phi:es^*s=e\}\subset\Phi. 
\end{equation}
Define $\alpha(s):A(s)\to\Phi$ by $\alpha(e)=f$, where $f$ satisfies (\ref{1}). Then $\alpha(s)$ is a bijection onto $B(s)=\alpha(A(s))\subset\Phi$ (the argument about uniqueness of $f$ works in the opposite direction, i.e. for $e$, too). Set $\alpha_s=(A(s),\alpha(s),B(s))$. Denote the set of all partial bijections on $\Phi$ by $PB(\Phi)$. Then $\alpha_s\in PB(\Phi)$ for any $s\in S_\Phi$. 

\begin{prop}
The map $s\mapsto\alpha_s$, $S_\Phi\to PB(\Phi)$, is a semigroup homomorphism.

\end{prop}
\begin{proof}
Obvious.

\end{proof}

\begin{example}
Let $S=PB(X_{2n})$, where $X_n=\{1,\ldots,n\}$, and let $\Phi=\{e,f\}$, where $e$ (resp., $f$) is the identity map of the set $Y_1=\{1,\ldots,n\}$ (resp., $Y_2=\{n+1,\ldots,2n\}$). Let $A,B\subset X_{2n}$ be two sets of the same cardinality, and let $s:A\to B$ be a partial bijection. Then $s\in S_\Phi$ if $A$ and $B$ coincide with either $Y_1$ or $Y_2$, hence $S_\Phi\cong PB(\{1,2\})$.

\end{example}

\begin{example}
Let $S=PB(X_n)$, $k<n$, and let $\Phi=\{e\}$, where $e$ is the identity map of the set $Y=\{1,\ldots,k\}$. Let $A,B\subset X_{n}$ be two sets of the same cardinality, and let $s:A\to B$ be a partial bijection. Then the condition $se=es$ implies that $|A\cap Y|=|B\cap Y|$, hence $s\in S_\Phi$ if $A$ and $B$ either contain $Y$ or do not intersect $Y$. In the first case $s$ is a permutation on $Y$ and a partial bijection on $X_n\setminus Y$, and in the second case $s$ is a partial bijection on $X_n\setminus Y$. Then $S_\Phi\cong (\{0\}\cup\Sigma_k)\times PB(X_{n-k})$, where $\Sigma_k$ denotes the transposition group on $k$ elements.

\end{example}

\section{Inverse semigroup from metrics on the double restricted to a family of subsets}

Let $X$, $Y$ be metric spaces. A map $f:Y\to X$ is almost isometric if there exists $C>0$ such that $|d(f(x),f(y))-d(x,y)|<C$ for any $x,y\in Y$. A self-map $h:X\to X$ is almost identity if there exists $C>0$ such that $d(x,h(x))<C$ for any $x\in X$. The two spaces, $Y_1$ and $Y_2$, are almost isometric if there exist almost isometries $f:Y_1\to Y_2$ and $g:Y_2\to Y_1$ such that $fg$ and $gf$ are almost identities for $Y_2$ and $Y_1$, respectively. 

Let $AI(Y,X)$ denote the set of all almost isometric maps from $Y$ to $X$. Let $\Psi\subset AI(Y,X)$, $\psi\in\Psi$, and let $A=\psi(Y)\subset X$. Let $\rho_A$ be a metric on the double of $X$ compatible with $d_X$ defined by (\ref{rho}).   

\begin{lem}
Let $\psi,\phi\in\Psi$, $A=\psi(Y)$, $B=\phi(Y)$. Then $[\rho_A]$ and $[\rho_B]$ are equivalent in $M(X)$. 

\end{lem}
\begin{proof}
By assumption, $A$ and $B$ are almost isometric, i.e. there exists $C>0$ and maps $f:A\to B$ and $g:B\to A$ such that $|d_X(f(x),f(y))-d_X(x,y)|<C$ for any $x,y\in A$, $|d_X(g(x),g(y))-d_X(x,y)|<C$ for any $x,y\in B$, and $d_X(x,f\circ g(x))<C$ for any $x\in B$ and $d_X(x,g\circ f(x))<C$ for any $x\in A$. Let $\rho$ be the metric on the double of $X$ campatible with $d_X$ defined by
$$
\rho(x,y')=\inf_{a\in A}d_X(x,a)+2C+d_X(f(a),y)
$$ 
It is easy to check the triangle inequality for $\rho$, whence it is really a metric, hence it defines some $s=[\rho]\in M(X)$. We have to show that $s^*s=[\rho_A]$ and $ss^*=[\rho_B]$. 
Recall that by Proposition 3.2 in \cite{M1}, two idempotent metrics, $\rho$ and $\sigma$, are equivalent if the two functions, $x\mapsto\rho(x,x')$ and $x\mapsto\sigma(x,x')$ are equivalent.
As
\begin{eqnarray*}
\rho^*\rho(x,x')&=&\inf_{y\in X}(2\inf_{a\in A}d_X(x,a)+2C+d_X(f(a),y)+1)\\
&=&2\inf_{a\in A}d_X(x,a)+2C+1,
\end{eqnarray*}
we have $s^*s=[\rho_A]$.
Similarly,
\begin{eqnarray*}
\rho\rho^*(x,x')&=&\inf_{y\in X}(2\inf_{a\in A}d_X(y,a)+2C+d_X(f(a),x)+1)\\
&=&2\inf_{a\in A}d_X(x,f(a))+2C+1=2d_X(x,f(A))+2C+1.
\end{eqnarray*}
As for any $b\in B$ there exists $a\in A$ such that $d_X(b,f(a))<C$, we have 
$$
|d_X(x,B)-d_X(f(A))|<C, 
$$
therefore, $|\rho\rho^*(x,x')-2d_X(x,B)|<4C+1$, hence $ss^*=[\rho_B]$.  

\end{proof}

\begin{lem}\label{lem:existence}
Let $\Phi=\{\psi(Y):\psi\in\Psi\}$, $\psi\in\Psi$, $A=\psi(Y)$ and $s\in M(X)_\Phi$. Suppose that $[\rho_A]s^*s\neq 0$. Then, by assumption, there exists $\phi\in\Psi$ such that $s[\rho_A]=[\rho_B]s$ and $[\rho_B]ss^*=[\rho_B]$, where $B=\phi(Y)$.
Let $\rho$ be a metric on the double of $X$ such that $[\rho]=s$. Then there exists $C>0$ and almost isometries $f:A\to B$ and $g:B\to A$ such that $\rho(x,f(x)')<C$ for any $x\in A$ and $\rho(x',g(x))<C$ for any $x\in B$.

\end{lem}
\begin{proof}
It follows from $se=fs$ that $ses^*=fss^*=f$. so we have $[\rho][\rho_A][\rho^*]=[\rho_B]$, i.e. $\rho\rho_A\rho^*$ is equivalent to $\rho_B$. This implies that there exists $C>0$ such that $\rho\rho_A\rho^*(x,x')<C/2$ for any $x\in B$.
Let $x\in B$. As 
\begin{eqnarray*}
\rho\rho_A\rho^*(x,x')&=&\inf_{u,v\in X}\rho^*(x,u')+\rho_A(u,v')+\rho(v,x')\\
&=&\inf_{u,v\in X}\rho(x',u)+\rho_A(u,v')+\rho(v,x')\\
&=&\inf_{u,v\in X,a\in A}\rho(x',u)+d_X(u,a)+1+d_X(a,v)+\rho(v,x')\\
&<&C/2,
\end{eqnarray*}
there exist $u\in X$, $a\in A$ such that $\rho(x',u)<C/2$ and $d_X(u,a)<C/2$. Set $g(x)=a$. Then $\rho(x',g(x))\leq\rho(x',u)+d_X(u,a)<C$ for any $x\in B$. Similarly, there is a map $f:A\to B$ such that $\rho(x,f(x)')<C$ for any $x\in A$.

\end{proof}

Set $\Phi=\{\rho_{\psi(Y)}:\psi\in\Psi\}$. Then $M(X)_\Phi$ is the inverse semigroup obtained from $M(X)$ by the procedure described in Section \ref{semigroup}. 

\section{Case of geodesic rays}

For convenience, here we consider {\it pointed} metric spaces $X$ with a basepoint $x_0\in X$. One of the most obvious choices for $Y$ is the half-line $Y=[0,\infty)$ with the natural metric, and with the basepoint 0. Let $\Psi$ denote the set of all isometries $\psi:[0,\infty)\to X$ with $\varphi(0)=x_0$. Then $\varphi([0,\infty))$ is a geodesic ray, so if we want to have enough elements in $\Psi$ then we should require that $X$ should be a {\it geodesic} space. Set 
$$
\Phi=\{[\rho_{\psi([0,\infty))}]:\psi\in\Psi\}\subset E(M(X)).
$$
We write $M(X)_r$ for $M(X)_\Phi$ with this choice for $\Phi$ (the subscript $r$ stands for `rays'). 

Recall some standard notions from metric geometry \cite{B-S}. A curve $\gamma:[0,1]\to X$ is rectifiable if 
$$
L(\gamma)=\sup\bigl\{\sum\nolimits_{i=1}^nd_X(\gamma(t_{i-1}),\gamma(t_i)):n\in\mathbb N, 0\leq t_0\leq t_1\leq\cdots\leq t_n\leq 1\bigr\}
$$ 
is finite. The metric $d_X$ is inner if $d_X(x,y)=\inf L(\gamma)$, where the infimum is taken over all rectifiable curves $\gamma$ from $x$ to $y$, for any $x,y\in X$. If this infimum is attained on some $\gamma$ then $\gamma$ is a geodesic connecting $x$ and $y$. 
A geodesic ray from $x_0\in X$ is an infinite length geodesic line that starts at $x_0$. A metric space is geodesic if any two points can be connected by a geodesic line.

Two geodesic rays, $r_1$ and $r_2$, are equivalent if there exists $C>0$ such that $d_X(r_1(t),r_2(t))<C$ for any $t\in[0,\infty)$. The set of all geodesic rays starting at the basepoint can be endowed with the following topology: the sets $V_{R,\varepsilon}(r)$ consisting of all geodesic rays $s$ such that $d_X(s(t),r(t))<\varepsilon$ for any $t\in[0,R)$ form its base at the point $r$. The corresponding quotient topology (with respect the equivalence above) gives a topology on the set $\partial X$ of equivalence classes of geodesic rays starting at the basepoint. The set $\partial X$ is called the visual boundary of $X$.

Let us consider the case when each $r\in\partial X$ is represented by a single geodesic ray. Then $\Phi=\{[\rho_r]:r\in\partial X\}$. In this case we can assign to each metric $\rho$ on the double of $X$ a function on $\partial X$ as follows.
For a geodesic ray $r\in\partial X$, set 
$$
f_\rho(r)=\inf\{C:\lim\sup_{x\in r}\rho(x,X')<C\}. 
$$
This gives a function $f_\rho:\partial X\to [1,\infty]$.

For $[\rho]\in M(X)_{r}$, consider the set $A=A([\rho])=f_\rho^{-1}([1,\infty))\subset\partial X$ (cf. (\ref{As})).

\begin{lem}
The set $A=f_\rho^{-1}([1,\infty))$ is an $F_\sigma$ set.

\end{lem}
\begin{proof}
Set $A_n=\{r\in\partial X: f_\rho(r)\leq n\}$. Let $r_0\in\overline{A}_n$, i.e. any neighborhood $V_{R,\varepsilon}(r_0)$ of $r_0$ intersects $A_n$. Suppose that $f_\rho(r_0)>n$. Then there exists $x\in r_0$ such that $\rho(x,X')>n$. Take $r\in A_n\cap V_{R,\varepsilon}(r_0)$, where $R=d_X(x_0,x)$ and $\varepsilon=\frac{\rho(x,X')-n}{2}$. Let $y\in r$, $d_X(x_0,y)=R$, then $d_X(x,y)<\varepsilon$ and $\rho(y,X')\leq n$. By the triangle inequality, 
$$
\rho(x,X')\leq d_X(x,y)+\rho(y,X')\leq \varepsilon+n<\rho(x,X'). 
$$
This contradiction shows that $x\in A_n$, hence $A_n$ is closed. Clearly, $A=f_\rho^{-1}([1,\infty))=\cup_{n\in\mathbb N}A_n$. 

\end{proof}

\section{Evaluation of $M(\mathbb R^n)_{r}$}\label{Rn}

Let $\mathbb R^n$ be the euclidean space with the standard metric. Then the geodesic rays starting at the basepoint can be identified with the sphere $\mathbb S^{n-1}=\partial\mathbb R^n$.

\begin{lem}\label{angle}
Let $s\in M(\mathbb R^n)_{r}$, $[\rho]=s$, and let $r_1, r_2$ be two rays in $\mathbb R^n$. Let $e_i=[\rho_{r_i}]$, $i=1,2$. Suppose that $e_is^*s\neq 0$, $i=1,2$, and let $f_1,f_2\in E(M(\mathbb R^n))$ satisfy $se_i=f_is$ and $f_iss^*=f_i$, $i=1,2$. Let $t_i$ be the ray such that $[\rho_{t_i}]=f_i$, which exists by Lemma \ref{lem:existence}. Let $d_{\mathbb S^{n-1}}$ be the standard metric on the sphere. Then $d_{\mathbb S^{n-1}}(t_1,t_2)=d_{\mathbb S^{n-1}}(r_1,r_2)$.

\end{lem}
\begin{proof}
By Lemma \ref{lem:existence}, there exists $C>0$ and an almost isometry $\varphi_i:r_i\to t_i$, $i=1,2$, such that $\rho(x,\varphi_i(x))<C$ for any $x\in r_i$. For $Y=[0,\infty)$, any almost isometry is close to the identity map, so without loss of generality we may assume that $\varphi_i$ is the bijective isometry between $r_i$ and $t_i$.

Let $x_i\in r_i$, $d_X(x_0,x_i)=R$, and let $y_i=\varphi_i(x_i)\in t_i$. Then $d_X(x_0,y_i)=R$. Let $\alpha$ (resp., $\beta$) be the angle between $r_1$ and $r_2$ (resp., between $t_1$ and $t_2$). Then 
$$
d_X(x_1,x_2)=2R\sin\frac{\alpha}{2}, \quad d_X(y_1,y_2)=2R\sin\frac{\beta}{2}. 
$$
As $\rho(x_i,y'_i)<C$, $|d_X(y_1,y_2)-d_X(x_1,x_2)|<2C$, hence $R|\sin\frac{\alpha}{2}-\sin\frac{\beta}{2}|<C$. As $R$ is arbitrary, $\alpha=\beta$.

\end{proof}

Let $PI(\mathbb S^{n-1})$ be the set of pairs $(A,u|_A)$, where $A\subset\mathbb S^{n-1}$ is an $F_\sigma$ set, and $u\in O(n)$ is an orthogonal operator on $\mathbb R^n$. This set has a natural injection to the inverse semigroup $PB(\mathbb S^{n-1})$ of all partial bijections of the sphere: assign to a pair $(A,u|_A)$ the partial bijection $u|_A:A\to B$, where $B=u|_A(A)$. Thus, $PI(\mathbb S^{n-1})$ is an inverse semigroup.

\begin{thm}
There exists a split semigroup surjection $M(\mathbb R^n)_{r}\to PI(\mathbb S^{n-1})$.

\end{thm}
\begin{proof}
Let $s\in M(\mathbb R^n)_{r}$, $[\rho]=s$. Set $A=\{r\in \mathbb S^{n-1}:f_\rho(r)<\infty\}$. For each $r\in A$ there exists $t\in\mathbb S^{n-1}$ such that $\rho(x,y')$ is uniformly bounded for any $x\in r$, $y\in t$ with $d_X(x_0,x)=d_X(x_0,y)$. By Lemma \ref{angle}, the map $r\mapsto t$ is an isometry of $\mathbb S^{n-1}$, hence is given by some orthogonal operator (may be not unique, but with unique restriction to $A$). This gives a map $\psi:M(\mathbb R^n)_{r}\to PI(\mathbb S^{n-1})$, which is well defined (i.e. does not depend on the representative $\rho$ in $s$). 

Let us construct a map in the opposite direction. Let $A=\cup_{m\in\mathbb N}A_m$, where $A_m$'s are closed subsets of $\mathbb S^{n-1}$ with $A_m\subset A_{m+1}$, $m\in\mathbb N$, let $u\in O(n)$, and let $B_m=u(A_m)$. Let $E_m=\{x\in\mathbb R^n:x\in r\mbox{\ for\ some\ }r\in A_m\}$. Set also $E_0=\{x_0\}$. For $x,y\in\mathbb R^n$, set
$$
\rho(x,y')=\inf_{m\in\mathbb N\cup\{0\}}\inf_{z\in E_m}[d_X(x,z)+m+1+d_X(u(z),y)].
$$
It is easy to check the triangle inequality, so $\rho$ determines a metric on the double of $\mathbb R^n$ compatible with the Euclidean metric. 

If $A=\cup_{m\in\mathbb N}A'_m$ with $A'_m$ closed and $A'_m\subset A'_{m+1}$, $m\in\mathbb N$, then, replacing $A'_m$ by $A_m\cap A'_m$, we may assume that $A'_m\subset A_m$, and as the sets $A_m$ are compact, there exists a monotone map $\kappa:\mathbb N\to\mathbb N$ such that $A_{\kappa(m)}\subset A'_m$ and $\lim_{m\to\infty}\kappa(m)=\infty$. Let $E'_m=\{x\in\mathbb R^n:x\in r\mbox{\ for\ some\ }r\in A'_m\}$, and let
$$
\rho'(x,y')=\inf_{m\in\mathbb N\cup\{0\}}\inf_{z\in E'_m}[d_X(x,z)+m+1+d_X(u(z),y)].
$$

Fix $x,y\in\mathbb R^n$. Take $\varepsilon>0$ and find $m\in\mathbb N$ and $z\in E_m$ such that 
$$
[d_X(x,z)+m+1+d_X(u(z),y)]-\rho(x,y')<\varepsilon.
$$
Then $z\in E'_{\tau(m)}$, where $\tau=\kappa^{-1}$, so there exists $k\in\mathbb N$ such that $m\leq k\leq \tau(m)$ such that $z\in E'_k\setminus E'_{k-1}$, and
$$
\rho'(x,y')\leq d_X(x,z)+k+1+d_X(z,y)<\rho(x,y')+(k-n)+1+\varepsilon,
$$
hence 
$$
\rho'(x,y')\leq \rho(x,y')+\tau(m)-m<\rho(x,y')+\tau(m).
$$
We also have $m\leq \rho(x,y')$, hence 
$$
\rho'(x,y')\leq \rho(x,y')+\tau(m)\leq \varphi(\rho(x,y')),
$$
where $\varphi(t)=t+\tau(t)$.

Similarly, we get $\rho(x,y')\leq \rho'(x,y')$ for any $x,y\in\mathbb R^n$, hence the metrics $\rho$ and $\rho'$ represent the same element in $M(\mathbb R^n)$. Let us check that this element lies in $M(\mathbb R^n)_{r}$. Let $r\notin A$. Set $E=\cup_{n\in\mathbb N}E_n$. As $A_n$ is closed for any $n\in\mathbb N$, $E_n$ is closed as well, hence for any $m\in\mathbb N$ there exists $R>0$ such that $d_X(x,E_m)>m$ for any $x\in r$ with $d_X(x_0,x)>R$. Then 
\begin{eqnarray*}
\rho(x,X')&=&\inf_{y\in X}\inf_{n\in\mathbb N}\inf_{z\in E_n\setminus E_{n-1}}[d_X(x,z)+n+d_X(u(z),y)]\\
&=&\inf_{n\in\mathbb N}\inf_{z\in E_n\setminus E_{n-1}}[d_X(x,z)+n].
\end{eqnarray*}
If $k\leq m$ and $z\in E_k$ then $d_X(x,E_k)\geq d_X(x,E_m)>m$, hence $d_X(x,z)+k>m$. If $k>m$ then $d_X(x,z)+k\geq n>m$. Thus, 
$$
\rho(x,X')=\inf_{k\in\mathbb N\cup\{0\}}\inf_{z\in E_k}[d_X(x,z)+k+1]>m\quad \mbox{when\ \ } d_X(x_0,x)>R,
$$ 
so in this case $[\rho_r\rho^*\rho]=0$. If $r\in A$ then $r\in A_k$ for some $k$. Set $t=u(r)$. If $x\in r$, $y\in t$ and $d_X(x_0,x)=d_X(x_0,y)$ then 
$$
\rho(x,y')=\inf_{l\in\mathbb N\cup\{0\}}\inf_{z\in E_l}[d_X(x,z)+l+1+d_X(u(z),u(x))]\leq l,
$$
hence $[\rho_r\rho]=[\rho\rho_t]$ and $[\rho_t\rho\rho^*]=[\rho_t]$, thus $[\rho]\in M(\mathbb R^n)_{r}$.

Assigning $[\rho]$ to $(A,u|_A)\in PI(\mathbb S^{n-1})$, we obtain a map $\chi:PI(\mathbb S^{n-1})\to M(\mathbb R^n)_{r}$. 
Direct check shows that $\psi\circ\chi$ is the identity map on $PI(\mathbb S^{n-1})$.

\end{proof}

\section{Evaluation of $M(X)_{r}$ for trees}

Let $X$ be a rooted tree with a basepoint $x_0$ (root), and let $d_X$ be the intrinsic metric on $X$ such that each edge has length one. A {\it ray} is a geodesic that starts at $x_0$ and has infinite length. We assume that for any point $x\in X$ there exists a ray that passes through $x$, i.e. $X$ has no dead ends. Rays can be described in terms of infinite sequences of subsequent vertices, $x_0,x_1,x_2,\ldots$ with $d_X(x_n,x_{n+1})=1$. Given a geodesic segment $\gamma$ connecting $x_0$ with a vertex $x$, denote by $U_\gamma$ the set of all rays that begin with $\gamma$.

Recall that the visual boundary $\partial X$ of $X$ is the set of all rays, endowed with the topology determined by the base consisting of all sets of the form $U_\gamma$, where $\gamma$ runs over the set of all geodesic segments starting at $x_0$.
The space $\partial X$ is known to be compact and Hausdorff. Moreover, this topology is metrizable: given two rays, $\lambda$ and $\mu$, with the common segment $\gamma$, set $d(\lambda,\mu)=e^{-|\gamma|}$, where $|\gamma|$ denotes the length of $\gamma$.    

Recall that a map $f:X\to Y$ between two metric spaces is {\it Lipschitz} if there exists $K>1$ such that $d_Y(f(x_1),f(x_2))\leq K d_X(x_1,x_2)$. If $X$ is locally compact, $X=\cup_{n\in\mathbb N} A_n$ with compact $A_n$'s then a map $f:X\to Y$ is {\it locally} Lipschitz if it is Lipschitz on each $A_n$ (with possibly increasing Lipschitz constants). Clearly, if $X=\cup_{n\in\mathbb N}B_n$ with compact $B_n$'s then a map is locally Lipschitz with respect to $X=\cup_{n\in\mathbb N}A_n$ iff it is locally Lipschitz with respect to $X=\cup_{n\in\mathbb N}B_n$. A bijective map $f:X\to Y$ between two locally compact metric spaces is locally {\it bi-}Lipschitz if both $f$ and $f^{-1}$ are locally Lipschitz.

For a metric space $Y$, a partial bijective locally bi-Lipschitz (PBLBL) map is a bijective locally bi-Lipschitz map $\varphi:A\to B$ between two $F_\sigma$ subspaces $A,B\subset Y$ (note that an $F_\sigma$ set is locally compact). It is easy to see that the set of all PBLBL maps has a natural structure of an inverse semigroup. Denote this inverse semigroup by $PBLBL(Y)$. 


\begin{lem}\label{rays}
Let $s\in M(X)_{r}$, $[\rho]=s$, and let $r_1, r_2$ be two rays in $X$. Let $e_i=[\rho_{r_i}]$, $i=1,2$. Suppose that $e_is^*s\neq 0$, $i=1,2$, and let $f_1,f_2\in E(M(X))$ satisfy $se_i=f_is$ and $f_iss^*=f_i$, $i=1,2$. Let $t_i$ be the ray such that $[\rho_{t_i}]=f_i$, which exists by Lemma \ref{lem:existence}. Let $d_{\partial X}$ be the metric on $\partial X$ described above. Then there exists $K>0$ such that $\frac{1}{K}<\frac{d_{\partial X}(t_1,t_2)}{d_{\partial X}(r_1,r_2)}<K$.

\end{lem}
\begin{proof}
By Lemma \ref{lem:existence}, there exists $C>0$ and an almost isometry $\varphi_i:r_i\to t_i$, $i=1,2$, such that $\rho(x,\varphi_i(x))<C$ for any $x\in r_i$. Any almost isometry of $Y=[0,\infty)$ is close to the identity map, so without loss of generality we may assume that $\varphi_i$ is the bijective isometry between $r_i$ and $t_i$.

Let $x_i\in r_i$, $d_X(x_0,x_i)=R$, and let $y_i=\varphi_i(x_i)\in t_i$, $i=1,2$. Then $d_X(x_0,y_i)=R$. Let $L$ (rep., $M$) be the length of the common part of the geodesic rays $r_1$ and $r_2$ (resp., $t_1$ and $t_2$). We assume that $R>L,M$. Then $d_X(x_1,x_2)=2(R-L)$, $d_X(y_1,y_2)=2(R-M)$. As $|d_X(y_1,y_2)-d_X(x_1,x_2)|<2C$, we have $|R-L-(R-M)|=|M-L|<C$. This means that $\frac{d_{\partial X}(t_1,t_2)}{d_{\partial X}(r_1,r_2)}=\frac{e^{-M}}{e^{-L}}<e^C$. Similarly, $\frac{d_{\partial X}(t_1,t_2)}{d_{\partial X}(r_1,r_2)}>e^{-C}$.

\end{proof}

\begin{thm}
Let $X$ be a rooted tree without dead ends, with the intrinsic metric. Then there exists a split semigroup surjection $M(X)_{r}\to PBLBL(\partial X)$.

\end{thm}
\begin{proof}

Let $s\in M(X)_{r}$, $[\rho]=s$. Set $A_m=\{r\in \partial X:f_\rho(r)<m\}$, $A=\cup_{m\in\mathbb N}A_m$. For each $r\in A$ there exists $t\in\partial X$ such that $\rho(x,y')$ is uniformly bounded for any $x\in r$, $y\in t$ with $d_X(x_0,x)=d_X(x_0,y)$. By Lemma \ref{rays}, the map $r\mapsto t$ is a Lipschitz map on each $A_m$, hence a locally bi-Lipschitz map on $A\subset\partial X$. This gives a map $\psi:M(\mathbb R^n)_{r}\to PBLBL(\partial X)$, which is well defined (i.e. does not depend on the representative $\rho$ in $s$).

Let us construct a map in the opposite direction. Let $A=\cup_{m\in\mathbb N}A_m$, where $A_m$'s are closed subsets of $\partial X$ with $A_m\subset A_{m+1}$, $m\in\mathbb N$, let $F:A\to\partial X$ be a locally bi-Lipschitz map, and let $B_m=F(A_m)$. Let $F|_{A_m}$ is a bi-Lipschitz map with a bi-Lipschitz contant $K_m=e^{C_m}$. Let $E_m=\{x\in X:x\in r\mbox{\ for\ some\ }r\in A_m\}$, $E=\cup_{m\in\mathbb N}E_m$. Set also $E_0=\{x_0\}$. 

For each $x\in X$ choose a geodesic ray $r_x\in\partial X$ such that $x\in r_x$. Let $x\in E_m$, and let $y\in F(r_x)$ satisfy $d_X(x_0,x)=d_X(x_0,y)$. Set $f(x)=y$. 

Another choice of geodesic rays, $x\mapsto r'_x$, gives rise to another map $f':E\to X$.

\begin{lem} One has
$|d_X(f(x),f(y))-d_X(x,y)|<2C_m$ for any $x,y\in E_m$, and $d_X(f(x),f'(x))<2C_m$ for any $x\in E_m$.

\end{lem}
\begin{proof}
Let $x,y\in E_m$.
Let $c$ be the length of the common part of $r_x$ and $r_y$, and let $a=d_X(x_0,x)$, $b=d_X(x_0,y)$. Then $d(r_x,r_y)=e^{-c}$, and $d_X(x,y)=a+b-2c$. 
Let $K_m$ be the Lipschitz constant on $A_m$, $C_m=\log K_m$, and let $t_x=F(r_x)$, $t_y=F(r_y)$. Then $\frac{e^{-c}}{K_m}<d(t_x,t_y)<e^{-c}K_m$, i.e. the length $d$ of the common part of $t_x$ and $t_y$ lies in the interval $(c-K_m,c+K_m)$. Then $d_X(f(x),f(y))=a+b-2d$, hence 
$$
|d_X(f(x),f(y))-d_X(x,y)|=2|c-d|=2C_m.
$$

If $x\in r_x,r'_x$ then the length of the common part of $r_x$ and $r'_x$ is not less than $c=d_X(x_0,x)$, hence $d(r_x,r'_x)\leq e^{-c}$. Let $t=F(r_x)$, $t'=F(r'_x)$. Then $d(t,t')<K_m e^{-c}=e^{-(c-C_m)}$, hence the length of the common part of $t$ and $t'$ is greater than $c-C_m$. As $d_X(x_0,f(x))=d_X(x_0,f'(x))=c$, $f(x)\in t$, $f'(x)\in t'$, we have $d_X(f(x),f'(x))<2C_m$.

\end{proof}

For $x,y\in X$, set
$$
\sigma(x,y')=\inf_{m\in\mathbb N\cup\{0\}}\inf_{z\in E_m}[d_X(x,z)+2C_m+d_X(u(z),y)].
$$

This formula determines a metric on $X\sqcup X$ (in fact, this is the intrinsic metric of the graph that is obtained from the two copies of the tree with added edges that connect any $x\in E_m$ with $f(x)'\in X'$ of length $2C_m$). Let us check that this metric is compatible with $d_X$. Let $x_1,x_2,y\in X$; $i=1,2$. Then
\begin{eqnarray*}
\sigma(x_1,y')+\sigma(x_2,y')&=&\inf_{m_i\in\mathbb N\cup\{0\}}\inf_{z_i\in E_{m_i}}d_X(x_1,z_1)+2C_{m_1}\\
&+&d_X(f(z_1),y)+d_X(x_2,z_2)+2C_{m_2}+d_X(f(z_2),y)\\
&\geq&\inf_{m_i\in\mathbb N\cup\{0\}}\inf_{z_i\in E_{m_i}}d_X(x_1,z_1)+2(C_{m_1}+C_{m_2})+d_X(f(z_1),f(z_2))\\
&\geq&\inf_{z_i\in E_{m_i}}d_X(x_1,z_1)+d_X(z_1,z_2)\\
&\geq&d_X(x_1,x_2),
\end{eqnarray*}
and, similarly, $\sigma(x,y'_1)+\sigma(x,y'_2)\geq d_X(y_1,y_2)$ for any $x,y_1,y_2\in X$.

As for Euclidean spaces (Section \ref{Rn}), it is easy to check that $[\sigma]\in M(X)_{r}$, that the map $\chi:PBLBL(\partial X)\to M(X)_{r}$, $\chi(A,F,B)=[\sigma]$, is well defined and that $\psi\circ\chi$ is the identity map on $PBLBL(\partial X)$.

\end{proof}


\begin{cor}
For $X=[0,\infty)$, $M(X)_0$ consists of zero and unit. For $Y=\mathbb R$, $M(Y)_0$ is the semigroup of partial bijections of the set $\{0,1\}$. 
\end{cor}
\begin{proof}
The space $X$ (resp., $Y$) has one end (resp., two ends). Thus, the semigroup of partial bijective locally bi-Lipschitz maps of the set of ends is the semigroup of partial bijections of the set consisting of one element (resp., of two elements).

\end{proof}

\end{document}